\documentclass[12pt]{amsart}

\usepackage{amsmath}	
\usepackage{amssymb}

\newtheorem{theoreme}{Th\'eor\`eme}[section]

\newtheorem{lemme}[theoreme]{Lemme}
\newtheorem{proposition}[theoreme]{Proposition}

\newtheorem{corollaire}[theoreme]{Corollaire}

\newcommand{\bC}{\mathbb{C}}

\newcommand{\bN}{\mathbb{N}}
\newcommand{\bP}{\mathbb{P}}
\newcommand{\bQ}{\mathbb{Q}}
\newcommand{\bR}{\mathbb{R}}
\newcommand{\bZ}{\mathbb{Z}}

\newcommand{\cF}{{\mathcal{F}}}

\newcommand{\cO}{{\mathcal{O}}}

\newcommand{\gA}{\mathfrak{A}}

\newcommand{\Hinf}{H_\infty}
\newcommand{\Hfin}{H_\mathrm{fin}}

\newcommand{\tP}{\tilde{P}}

\newcommand{\ualpha}{\boldsymbol{\alpha}}

\newcommand{\uD}{\mathbf{D}}

\newcommand{\um}{{\mathbf{m}}}
\newcommand{\un}{{\mathbf{n}}}

\newcommand{\ur}{{\mathbf{r}}}

\newcommand{\uu}{\mathbf{u}}
\newcommand{\uU}{\mathbf{U}}
\newcommand{\uun}{\mathbf{1}}
\newcommand{\uw}{\mathbf{w}}

\newcommand{\usX}{\underline{\uX}}
\newcommand{\uX}{\mathbf{X}}

\newcommand{\uz}{\mathbf{z}}

\newcommand{\bCC}{\bC[\usX]_{D_0} \times \cdots \times \bC[\usX]_{D_t}}

\newcommand{\et}{\quad\mbox{et}\quad}
\newcommand{\fleche}{\longrightarrow}
\newcommand{\Ga}{\mathbb{G}_{\mathrm{a}}}
\newcommand{\Gm}{\mathbb{G}_{\mathrm{m}}}
\newcommand{\malpha}{\um\cdot\ualpha}
\newcommand{\mj}{{\substack{\um\in\Sigma(S)\\ 0\le j<T}}}
\newcommand{\nalpha}{\un\cdot\ualpha}

\newcommand{\ou}{\quad\mbox{ou}\quad}
\newcommand{\Qbar}{\overline{\bQ}}
\newcommand{\ralpha}{\ur\cdot\ualpha}
\newcommand{\Res}{\mathrm{Res}}

\DeclareMathOperator{\cont}{cont}

\begin{document}

\baselineskip=15pt 

\title[Sur le th\'eor\`eme de Lindemann-Weierstrass effectif]
{Une version effective du th\'eor\`eme
de Lindemann-Weierstrass par des m\'ethodes
d'ind\'ependance alg\'ebrique}
\author{Damien ROY}

\address{
   D\'epartement de Math\'ematiques\\
   Universit\'e d'Ottawa\\
   585 King Edward\\
   Ottawa, Ontario K1N 6N5, Canada}
\email{droy@uottawa.ca}
\thanks{Recherche partiellement support\'ee par le CRSNG}

\renewcommand{\abstractname}{R\'esum\'e}
\begin{abstract}
On pr\'esente une nouvelle d\'emonstration compl\`etement effective du th\'eor\`eme de Lindemann-Weierstrass bas\'ee sur des m\'ethodes d'ind\'ependance alg\'ebrique.  Quoique sensiblement moins bonne que la meilleure estimation connue \`a ce jour, d\^ue \`a A.~Sert, elle am\'eliore toutefois la meilleure estimation, d\^ue \`a M.~Ably, obtenue par ce type de m\'ethode.  La nouveaut\'e de l'argument r\'eside dans la simplicit\'e de la construction de fonctions auxiliaires.  On exploite ce trait pour introduire le non-sp\'ecialiste aux m\'ethodes d'ind\'ependance alg\'ebrique.
\par
\medskip
\noindent
{\sc Abstract.}
We present a new completely effective proof of the Lindemann-Weierstrass theorem based on algebraic independence methods.  Although it is slightly weaker than the best known estimate due to A.~Sert, it improves the best estimate due to M.~Ably obtained by such methods.  The novelty of the proof lies in the simplicity of the construction of auxiliary functions, a fact that we exploit to introduce the non-specialist to methods of algebraic independence.
\end{abstract}

\maketitle

%
%

\section{Introduction}
\label{sec:Intro}

Le th\'eor\`eme de Lindemann-Weierstrass est un des r\'esultats les plus satisfaisants de la th\'eorie des nombres transcendants.  Il peut se formuler de mani\`eres \'equivalentes soit comme un \'enonc\'e d'ind\'ependance lin\'eaire:
\begin{quote}
``Si $\beta_1,\dots,\beta_N \in \bC$ sont des nombres alg\'ebriques distincts, alors $e^{\beta_1},\dots,e^{\beta_N}$ sont lin\'eairement ind\'ependants sur $\bQ$.'',
\end{quote}
ou bien comme un \'enonc\'e d'ind\'ependance alg\'ebrique:
\begin{quote}
``Si $\alpha_1,\dots,\alpha_t \in \bC$ sont des nombres alg\'ebriques lin\'eairement ind\'ependants sur $\bQ$, alors $e^{\alpha_1},\dots,e^{\alpha_t}$ sont alg\'ebriquement ind\'ependants sur $\bQ$.''.
\end{quote}
La preuve originale \cite{LW}, \'etablie par K.~Weierstrass en 1885, repose sur la premi\`ere formulation gr\^ace \`a une extension de la m\'ethode de C.~Hermite \cite{H}.  Elle emploie des syst\`emes d'approximants de Pad\'e simultan\'es pour les fonctions $e^{\beta_1x},\dots,e^{\beta_Nx}$. Des variations de la preuve, utilisant la m\^eme m\'ethode tout en y apportant diff\'erentes simplifications ou des \'eclairages nouveaux, ont \'et\'e propos\'ees par plusieurs auteurs dont D.~Hilbert et K.~Mahler.  L'appendice de \cite{Ma2} en fournit un compte-rendu exhaustif.  On pourra consulter les chapitres 1 de \cite{Ba} ou de \cite{Ne2} pour une d\'emonstration tr\`es succincte du th\'eor\`eme ou encore \cite{W1983} pour une pr\'esentation motiv\'ee de la m\'ethode d'Hermite.

En 1929, C.~L.~Siegel a initi\'e une vaste g\'en\'eralisation de ces recherches en introduisant sa th\'eorie des E-fonctions, compl\'et\'ee par A.~B.~Shidlovski\u{\i} dans une s\'erie de travaux \`a partir de 1954.  Cette th\'eorie et certains des r\'esultats auxquels elle conduit sont expos\'es dans \cite[ch.~4--7]{Ma2}, \cite[ch.~2]{Ne2} et \cite[ch.~11]{Ba}.  N\'eanmoins, il s'agit encore, \`a la base,  de m\'ethodes d'ind\'ependance lin\'eaire.

La premi\`ere d\'emonstration du th\'eor\`eme de Lindemann-Weierstrass bas\'ee sur la seconde formulation, utilisant donc des m\'ethodes d'ind\'ependance alg\'ebrique, a \'et\'e pr\'esent\'ee en 1980 par G.~V.~Chudnovsky dans \cite{Chud} pour au plus trois nombres (i.e.\ pour $t\le 3$). Par une adaptation de cette m\'ethode il obtient aussi un analogue de ce th\'eor\`eme pour la fonction $\wp$ de Weierstrass associ\'ee \`a une courbe elliptique d\'efinie sur $\Qbar$ avec multiplication complexe, mais avec la m\^eme limitation, lev\'ee par la suite par P.~Philippon \cite{Ph1} et G.~W\"ustholz \cite{Wu}.

Le but de ce travail est de donner une nouvelle d\'emonstration simple du th\'eor\`eme de Lindemann-Weierstrass, \`a l'aide de m\'ethodes d'ind\'ependance alg\'ebrique, sous une forme quantitative am\'eliorant celle obtenue par M.~Ably en 1994 par des m\'ethodes semblables \cite{Ably}.  Notre r\'esultat principal, ci-dessous, est toutefois moins pr\'ecis que la meilleure estimation pleinement explicite, d\^ue \`a A.~Sert \cite{Sert}.

\begin{theoreme}
\label{thm:eff}
Soient $\alpha_1,\dots,\alpha_t \in \bC$ des nombres alg\'ebriques lin\'eairement ind\'ependants sur $\bQ$.  Soit $c$ un majorant des valeurs absolues de tous leurs conjugu\'es, soit $q\in\bN$ un entier positif tel que $q\alpha_1, \dots, q\alpha_t$ soient des entiers alg\'ebriques, et soit $d$ le degr\'e de $\bQ(\alpha_1, \dots, \alpha_t)$ sur $\bQ$.  Pour toute paire d'entiers positifs $D$ et $H$ et tout polyn\^ome non nul $P \in \bZ[X_1, \dots, X_t]$ de degr\'e au plus $D$ \`a coefficients entiers en valeur absolue au plus $H$, on a
\[
 |P(e^{\alpha_1},\dots,e^{\alpha_t})| \ge H^{-3dS^t}\exp\left(-(cqS)^{18S^t}\right),
\]
o\`u $S=6dt(t!)D$.
\end{theoreme}

La borne inf\'erieure donn\'ee par le th\'eor\`eme est une \emph{mesure d'ind\'ependance alg\'ebrique}.  Elle implique l'ind\'ependance alg\'ebrique de $e^{\alpha_1},\dots,e^{\alpha_t}$ sur $\bQ$. Le premier r\'esultat de ce type revient \`a Mahler \cite{Ma1} en 1931.  En s'appuyant sur la m\'ethode d'Hermite, il d\'emontre, avec les notations ci-dessus, l'existence de quantit\'es $c_1 = c_1(\alpha_1,\dots,\alpha_t)$ et $H_0=H_0(D,\alpha_1,\dots,\alpha_t)$, non explicit\'ees, telles que $|P(e^{\alpha_1},\dots,e^{\alpha_t})| \ge H^{-c_1D^t}$ si $H\ge H_0$.  En 1977, Yu.~V.~Nesterenko d\'emontre une version quantitative du th\'eor\`eme de Siegel-Shidlovski\u{\i}.  Dans la situation pr\'esente, son r\'esultat \cite[thm.~4]{Ne1} fournit $c_1=(4d)^t(td^2+d+1)$ et montre qu'on peut prendre $H_0=\exp(\exp(c_2D^{2t}\log(D+1)))$ pour une constante $c_2=c_2(\alpha_1,\dots,\alpha_t)$ non explicit\'ee.  Le r\'esultat d'Ably mentionn\'e plus t\^ot fait moins bien au niveau de la constante $c_1$ mais am\'eliore la d\'ependance de $H_0$ en fonction du majorant $D$ du degr\'e de $P$.  Sans entrer dans les d\'etails, disons simplement que, dans les notations du th\'eor\`eme ci-dessus, la minoration qu'il obtient revient \`a prendre $S$ de l'ordre de $d^t 2^{4t^2}D$.  Quant \`a la mesure obtenue par Sert, elle revient plut\^ot \`a prendre $S$ de l'ordre de $dtD$.  Pour y parvenir, ce dernier utilise la remarquable m\'ethode des d\'eterminants d'interpolation de M.~Laurent.  Il obtient d'abord une version quantitative de la premi\`ere forme du th\'eor\`eme de Lindemann-Weierstrass, donc une mesure d'ind\'ependance lin\'eaire, puis, par sp\'ecialisation, il en d\'eduit une mesure d'ind\'ependance alg\'ebrique.  Dans ce contexte, mentionnons qu'\`a la suite du travail de Sert, D.~Bertrand a utilis\'e \`a son tour la m\'ethode des d\'eterminants d'interpolation pour donner une nouvelle d\'emonstration du th\'eor\`eme de Siegel-Shidlovski\u{\i} \cite{Be} qui, pour reprendre les mots de son auteur, jette un pont entre les d\'emonstrations originales de ce th\'eor\`eme et celle plus r\'ecente obtenue par Y.~Andr\'e \cite{An} g\'en\'eralisant l'approche d\'eploy\'ee par J.-P.~B\'ezivin et P.~Robba dans leur preuve ad\'elique du th\'eor\`eme de Lindemann-Weierstrass (version lin\'eaire) \cite{BR}.  Quoique les d\'emonstrations de \cite{An} et de \cite{BR} ne semblent pas conduire ais\'ement \`a des \'enonc\'es quantitatifs, il est probable que la preuve du th\'eor\`eme de Lindemann-Weierstrass que donne D.~Bertrand dans \cite[\S5]{Be} puisse quant \`a elle mener \`a une nouvelle mesure d'ind\'ependance alg\'ebrique des exponentielles de nombres alg\'ebriques.

Sur un plan plus sp\'eculatif, on note que le th\'eor\`eme de Lindemann-Weierstrass fournit une d\'emontration indirecte des conjectures de \cite{R2001} pour les familles de nombres de la forme $(\alpha_1, \dots, \alpha_t, e^{\alpha_1}, \dots, e^{\alpha_t})$ avec $\alpha_1, \dots, \alpha_t$ alg\'ebriques sur $\bQ$.  Cependant, il est possible qu'une d\'emonstration directe puisse conduire \`a d'autres mesures d'ind\'ependance alg\'ebrique comme on en conna\^{\i}t dans le cas d'un seul nombre.  Par exemple, P.~L.~Cisjouw \cite{Ci} a obtenu en 1974 la mesure $|P(e^{\alpha_1})|\ge \exp(-c_3D^2(D+\log H))$ avec $c_3=c_3(\alpha_1)$ qui, pour les petites valeurs de $H$, est plus pr\'ecise que les mesures d\'ecrites ci-dessus.

Au niveau des outils, la preuve du th\'eor\`eme \ref{thm:eff} utilise seulement la notion de hauteur de Weil d'un point alg\'ebrique et quelques propri\'et\'es fondamentales du r\'esultant de polyn\^omes en plusieurs variables rappel\'ees aux paragraphes \ref{sec:outils} et \ref{sec:Res}.  Dans ses grandes lignes, le sch\'ema de d\'emonstration est classique.  Le but est de prendre le r\'esultant de l'homog\'en\'eis\'e $^hP$ de $P$ avec des polyn\^omes homog\`enes $Q_1,\dots,Q_t$ \`a coefficients dans $K=\bQ(\alpha_1,\dots,\alpha_t)$, construits de telle sorte que leurs valeurs absolues au point $\uu = (1, e^{\alpha_1}, \dots, e^{\alpha_t})$ soient petites et que $^hP, Q_1, \dots, Q_t$ n'aient pas de z\'ero commun dans $\bP^t(\bC)$.  Alors le r\'esultant de ces polyn\^omes n'est pas nul et, sous des conditions appropri\'ees, cela conduit \`a une minoration pour la valeur absolue de $^hP(\uu)=P(e^{\alpha_1}, \dots, e^{\alpha_t})$.  Pour atteindre ce but, notre premier pas r\'eside en la construction d'une fonction analytique parti\-cu\-li\`erement simple, d\'ecrite au paragraphe suivant.  Ses valeurs fournissent une famille $\cF$ de polyn\^omes homog\`enes de degr\'e $S$, \`a coefficients dans $K$, dont on montre au paragraphe \ref{sec:LZ} qu'ils n'ont pas de z\'ero commun dans $\bP^t(\bC)$.  On en tire $Q_1,\dots,Q_t$, et on conclut, au paragraphe \ref{sec:Est}, avec la preuve d'une version l\'eg\`erement plus g\'en\'erale du th\'eor\`eme \ref{thm:eff}.

%
%

\section{Fonction auxiliaire}
\label{sec:Fa}

On utilise le r\'esultat d'interpolation suivant qui remonte \`a Lagrange.  Pour tout sous-ensemble fini $E$ de $\bC$ de cardinalit\'e $N\ge 1$, tout entier $T\ge 1$, et toute famille de nombres complexes $u_{\alpha,j}$ index\'ee par les couples $(\alpha,j)$ avec $\alpha\in E$ et $j\in\{0,1,\dots,T-1\}$, il existe un et un seul polyn\^ome $p(x)$ de $\bC[x]$ de degr\'e $< NT$ tel que
\[
 p^{(j)}(\alpha)=u_{\alpha,j} \quad (\alpha\in E,\, 0\le j<T),
\]
o\`u $p^{(j)}$ d\'esigne la $j$-i\`eme d\'eriv\'ee de $p$.  En particulier, si $f$ est une fonction holomorphe d\'efinie sur un ouvert de $\bC$ contenant $E$, on peut choisir $p$ de telle sorte que
\[
 p^{(j)}(\alpha)=f^{(j)}(\alpha) \quad (\alpha\in E,\, 0\le j<T).
\]
Alors la diff\'erence $f(x)-p(x)$ s'annule avec multiplicit\'e au moins $T$ en chaque point de $E$.

Soient $\alpha_1,\dots,\alpha_t$, $c$ et $q$ comme dans l'\'enonc\'e du th\'eor\`eme \ref{thm:eff}.  Pour notre objet, nous employons un ensemble $E$ qui d\'epend d'un param\`etre entier $S\ge 1$, et qui est contenu dans le sous-groupe de $\bC$ engendr\'e par les coordonn\'ees du point $\ualpha=(\alpha_1,\dots,\alpha_t)$.  En posant, pour chaque $\um=(m_1,\dots,m_t)\in\bZ^t$,
\[
 \um\cdot\ualpha = m_1\alpha_1+\cdots+m_t\alpha_t
 \et
 |\um|=|m_1|+\cdots+|m_t|,
\]
l'ensemble en question est
\begin{equation}
 \label{Fa:eq:E}
 E=\{ \um\cdot\ualpha \,;\, \um\in\Sigma(S)\}
 \quad\text{o\`u}\quad
 \Sigma(S)=\{ \um\in\bN^t \,;\, |\um|<S \}.
\end{equation}
Nous nous proposons d'\'etudier la fonction
\[
 g(x)=e^x-p(x)
\]
o\`u $p(x)$ est le polyn\^ome d'interpolation construit comme ci-dessus pour le choix de $f(x)=e^x$.  Cette fonction auxiliaire d\'epend donc de deux param\`etres entiers $S,T\ge 1$.  On verra que, pour des choix appropri\'es de ceux-ci, les d\'eriv\'ees de $g(x)$ aux points $\um\cdot\ualpha$ avec $\um\in\Sigma(S+1)$ sont si petites que cela interdit toute relation de d\'ependance alg\'ebrique entre $e^{\alpha_1}, \dots, e^{\alpha_t}$.  Cela conduira \`a la mesure d'ind\'ependance alg\'ebrique pr\'esent\'ee dans l'introduction.

Pour \'etablir qu'une telle fonction est petite en de pareils points, on pense d'abord \`a employer un lemme de Schwarz. Or, cela ne fonctionne pas ici, puisque la majoration que l'on poss\`ede pour le degr\'e du polyn\^ome $p(x)$ est essentiellement \'egale au nombre de z\'eros $T|\Sigma(S)|$ que $g(x)$ acquiert, compte tenu des multiplicit\'es.  Nous allons donc proc\'eder autrement.

Pour $S,T\ge 1$ fix\'es, on pose $N=|\Sigma(S)|$. Alors $E$ contient $N$ points en vertu de l'hypoth\`ese d'ind\'ependance lin\'eaire de $\alpha_1,\dots,\alpha_t$.  Pour chaque $\um\in\bN^t$ avec $|\um|<S$ et chaque $j\in\{0,1,\dots,T-1\}$, on d\'esigne par $A_{\um,j}(x)$ le polyn\^ome de $\bC[x]$ de degr\'e $<NT$ qui satisfait
\begin{equation}
 \label{Fa:eq:A}
 A_{\um,j}^{(\ell)}(\un\cdot\ualpha)
 =\begin{cases}1 &\text{si $\un=\um$ et $\ell=j$,}\\0 &\text{sinon,}\end{cases}
\end{equation}
pour tout $\un\in\Sigma(S)$ et tout $\ell\in\{0,1,\dots,T-1\}$.  On note aussi $\bC\{x\}$ l'anneau des fonctions holomorphes sur $\bC$ qu'on identifie au sous-anneau de $\bC[[x]]$ des s\'eries convergentes sur tout $\bC$.  La fonction exponentielle $e^x$ est un \'el\'ement de cet anneau et la fonction auxiliaire $g(x)$ introduite ci-dessus est son image sous l'application lin\'eaire $\varphi\colon \bC\{x\} \to \bC\{x\}$ donn\'ee par
\begin{equation}
 \label{Fa:eq:phi}
 \varphi(f(x)) = f(x) - \sum_\mj f^{(j)}(\um\cdot\ualpha)A_{\um,j}(x).
\end{equation}
Le noyau de $\varphi$ \'etant le sous-espace $\bC[x]_{<NT}$ des polyn\^omes de degr\'e $<NT$, on en d\'eduit le r\'esultat suivant o\`u, pour $j\in \bN$ et $k\in \bR$, on pose $k^{(j)}=k(k-1)\dots(k-j+1)$.

\begin{proposition}
\label{Fa:prop:f}
Pour toute s\'erie enti\`ere $f(x)=\sum_{k=0}^\infty c_k x^k \in \bC\{x\}$, on a
\[
 \varphi(f(x))
   = \sum_{k=NT}^\infty c_k \varphi(x^k)
   = \sum_{k=NT}^\infty c_k \bigg( x^k - \sum_\mj k^{(j)} (\um\cdot\ualpha)^{k-j} A_{\um,j}(x) \bigg),
\]
les s\'eries de droite convergeant uniform\'ement sur tout compact de $\bC$.
\end{proposition}

\begin{proof}[D\'emonstration]
On sait que, pour tout $j\in\bN$, on a $f^{(j)}(x) = \sum_{k=j}^\infty c_k k^{(j)} x^{k-j}$, la convergence de cette s\'erie \'etant uniforme sur tout compact.  On en d\'eduit, par r\'earrangement des termes,
\[
 \varphi(f(x))
 = \sum_{k=0}^\infty c_k x^k
   - \sum_\mj \bigg( \sum_{k=j}^\infty c_k k^{(j)} (\um\cdot\ualpha)^{k-j} \bigg) A_{\um,j}(x)
  = \sum_{k=0}^\infty c_k \varphi(x^k).
\]
La conclusion suit puisque $\varphi(x^k)=0$ pour $k<NT$.
\end{proof}

Puisque la convergence est uniforme sur tout compact, on peut d\'eriver termes \`a termes la s\'erie qui repr\'esente $\varphi(f(x))$.  Pour notre fonction auxiliaire, on en d\'eduit les formules suivantes dont la premi\`ere d\'ecoule de la d\'efinition et la seconde de la proposition.

\begin{corollaire}
\label{Fa:cor:g}
La fonction auxiliaire $g(x) = \varphi(e^x) = e^x - \sum e^{\um\cdot\ualpha}A_{\um,j}(x)$
satisfait
\begin{align*}
 g^{(\ell)}(\un\cdot\ualpha)
   &= e^{\un\cdot\ualpha} - \sum_\mj e^{\um\cdot\ualpha}A^{(\ell)}_{\um,j}(\un\cdot\ualpha)\\
   &= \sum_{k=NT}^\infty \frac{1}{k!} \bigg( k^{(\ell)}(\un\cdot\ualpha)^{k-\ell} - \sum_\mj k^{(j)} (\um\cdot\ualpha)^{k-j} A^{(\ell)}_{\um,j}(\un\cdot\ualpha) \bigg)
\end{align*}
pour tout $\ell\in\bN$ et tout $\un\in\bZ^t$.
\end{corollaire}

Comme les polyn\^omes $A_{\um,j}(x)$ sont \`a coefficients dans $\bQ(\alpha_1,\dots,\alpha_t)$, la premi\`ere formule montre que les nombres $g^{(\ell)}(\un\cdot\ualpha)$ appartiennent au corps $\bQ(\alpha_1,\dots,\alpha_t,e^{\alpha_1},\dots,e^{\alpha_t})$.  La seconde permet de majorer leurs valeurs absolues gr\^ace au fait que $k!$ cro\^{\i}t tr\`es rapidement en fonction de $k$.  C'est fondamentalement le m\^eme ph\'enom\`ene qui est \`a la base de la m\'ethode d'Hermite.  En pratique, on prendra $\un\in\Sigma(S+1)$ et $\ell\in\{0,1,\dots,T-1\}$.  Les deux paragraphes suivants sont un rappel d'outils qui permettront d'exploiter ces donn\'ees.

%
%

\section{Hauteur}
\label{sec:outils}

Soit $K\subset \bC$ une extension alg\'ebrique de $\bQ$ de degr\'e fini $d$, soit $\cO_K$ l'anneau des entiers de $K$, et soit $N_{K/\bQ}$ la norme de $K$ sur $\bQ$.  Pour chaque entier $n\ge 1$ et chaque point non nul $\uu=(u_0,u_1,\dots,u_n)\in K^{n+1}$, on d\'esigne par $N_{K/\bQ}(\uu)$ la norme de l'id\'eal fractionnaire de $K$ engendr\'e par $u_0,\dots,u_n$ en tant que $\cO_K$-module.  Pour chaque plongement $\sigma$ de $K$ dans $\bC$, on d\'esigne aussi par $\uu^\sigma$ le point $(\sigma(u_0),\dots,\sigma(u_n))\in\bC^{n+1}$ et par $\|\uu^\sigma\| = \max_{0\le i\le n} |\sigma(u_i)|$ sa norme du maximum.  Enfin, on pose
\[
 \Hinf(\uu)=\prod_\sigma \|\uu^\sigma\|^{1/d}, \quad
 \Hfin(\uu)=N_{K/\bQ}(\uu)^{-1/d} \et
 H(\uu)=\Hinf(\uu)\Hfin(\uu),
\]
le produit de gauche portant sur les $d$ plongements distincts $\sigma$ de $K$ dans $\bC$.  Le nombre $H(\uu)$ s'appelle la \emph{hauteur de Weil absolue} de $\uu$.  Elle satisfait $H(\uu)\ge 1$ et $H(\lambda\uu)=H(\uu)$ pour tout $\lambda\in K^*$.  De plus, si $L$ est un sous-corps de $K$ contenant les coordonn\'ees de $\uu$, alors la hauteur de $\uu$ est la m\^eme calcul\'ee sur $K$ ou sur $L$.   On renvoie le lecteur \`a \cite[\S 3.2]{W2000} pour une d\'efinition alternative de la hauteur en termes de produit de hauteurs locales et un traitement plus complet de ses propri\'et\'es.

On d\'efinit la \emph{hauteur} $H(P)$ d'un polyn\^ome non nul $P$ \`a coefficients dans $K$ comme \'etant la hauteur du vecteur de ses coordonn\'ees dans un ordre quelconque.  En particulier, si $P$ est \`a coefficients dans $\bZ$, alors $H(P)=\cont(P)^{-1}\|P\|$, o\`u $\cont(P)$ d\'esigne le contenu de $P$, c'est-\`a-dire le pgcd de ses coefficients.  Par ailleurs, on d\'efinit la \emph{longueur} $L(P)$ d'un polyn\^ome $P$ \`a coefficients dans $\bC$ comme la somme des valeurs absolues de ses coefficients.  Ces notions nous serons utiles \`a travers le r\'esultat suivant.

\begin{proposition}
\label{propU}
Soient $\uU_i=(U_{i,1},\dots,U_{i,n_i})$ pour $i=0,\dots,t$, des familles ind\'ependantes d'ind\'etermin\'ees sur $K$, et soit $R\in \bZ[\uU_0,\dots,\uU_t]$ un polyn\^ome multi-homog\`ene de multi-degr\'e $(N_0,\dots,N_t)$ en ces familles d'ind\'etermin\'ees.  Supposons que, pour $i=0,\dots,t$, il existe $\uu_i\in K^{n_i}$ et $\epsilon_i\in\bC^{n_i}$ tels que
\[
 R(\uu_0,\dots,\uu_t)\neq 0 \et R(\uu_0+\epsilon_0,\dots,\uu_t+\epsilon_t)=0.
\]
Alors on a
\[
 1 \le 2^{N_0+\cdots+N_t} L(R)^d  H(\uu_0)^{dN_0}\cdots H(\uu_t)^{dN_t}
        \max\left( \frac{\|\epsilon_0\|}{\|\uu_0\|}, \dots, \frac{\|\epsilon_t\|}{\|\uu_t\|} \right).
\]
\end{proposition}

On peut consid\'erer cette derni\`ere in\'egalit\'e comme une minoration de la distance du point $(\uu_0,\dots,\uu_t)$ au plus proche z\'ero complexe de $R$.

\begin{proof}[D\'emonstration]
Pour $i=0,\dots,t$, d\'esignons par $\gA_i$ l'id\'eal fractionnaire de $K$ engendr\'e par les coordonn\'ees de $\uu_i$. Comme $R$ est multi-homog\`ene de multi-degr\'e $(N_0,\dots,N_t)$, le nombre
$R(\uu_0,\dots,\uu_t)$ est un \'el\'ement non nul de $\gA_0^{N_0}\cdots\gA_t^{N_t}$ et par suite
\begin{equation}
\label{propU:eq}
 |N_{K/\bQ}(R(\uu_0,\dots,\uu_t))|
  \ge N_{K/\bQ}(\gA_0^{N_0}\cdots\gA_t^{N_t})
  = \Hfin(\uu_0)^{-dN_0}\cdots\Hfin(\uu_t)^{-dN_t}.
\end{equation}
Pour tout plongement $\sigma\colon K\to \bC$, on a aussi
\[
 |\sigma(R(\uu_0,\dots,\uu_t))|
 = |R(\uu_0^\sigma,\dots,\uu_t^\sigma)|
 \le L(R) \|\uu_0^\sigma\|^{N_0} \cdots \|\uu_t^\sigma\|^{N_t}.
\]
Pour le plongement donn\'e par l'inclusion de $K$ dans $\bC$, on utilise plut\^ot l'estimation
\begin{align*}
 |R(\uu_0,\dots,\uu_t)|
 &= |R(\uu_0+\epsilon_0,\dots,\uu_t+\epsilon_t)-R(\uu_0,\dots,\uu_t)|\\
 &\le 2^{N_0+\cdots+N_t} L(R) \|\uu_0\|^{N_0}\cdots\|\uu_t\|^{N_t}
    \max\left( \frac{\|\epsilon_0\|}{\|\uu_0\|}, \dots, \frac{\|\epsilon_t\|}{\|\uu_t\|} \right)
\end{align*}
qui se v\'erifie ais\'ement en se ramenant d'abord au cas o\`u $R$ est un mon\^ome puis, par homog\'en\'eit\'e, au cas o\`u $\|\uu_0\|=\cdots=\|\uu_t\|=1$ (comme on a alors $|R(\uu_0,\dots,\uu_t)| \le L(R)$, on peut, pour ce dernier calcul, supposer que $\|\epsilon_0\|,\dots,\|\epsilon_t\|\le 1$).   On en d\'eduit que
\[
 |N_{K/\bQ}(R(\uu_0,\dots,\uu_t))|
 \le 2^{N_0+\cdots+N_t} L(R)^d \Hinf(\uu_0)^{dN_0}\cdots\Hinf(\uu_t)^{dN_t}
     \max_{0\le i\le t} \frac{\|\epsilon_i\|}{\|\uu_i\|},
\]
et la conclusion suit en combinant cette estimation avec \eqref{propU:eq}.
\end{proof}

%
%

\section{R\'esultant}
\label{sec:Res}

L'introduction du r\'esultant comme outil de la th\'eorie des nombres transcendants est rela\-tivement r\'ecente. Quoique E.~Borel \cite{Bo} ait \'et\'e le premier \`a l'utiliser pour des polyn\^omes en une variable en 1899, c'est seulement \`a partir de 1949 que, gr\^ace au crit\`ere de Gel'fond, son emploi s'est syst\'ematis\'e. Pour les polyn\^omes en plusieurs variables, il faut encore attendre les travaux de G.~V.~Chudnovsky autour de 1974, puis l'introduction de la th\'eorie de l'\'elimination par Yu.~Nesterenko \cite{Ne1} en 1977.  De nos jours, ces efforts culminent avec le crit\`ere de P.~Philippon \cite{Ph2} de 1986 qui, pour les questions d'ind\'ependance alg\'ebrique, fournit une g\'en\'eralisation quasi-optimale du crit\`ere de Gel'fond (voir l'introduction de \cite{Ph2} pour un bref historique des travaux ant\'erieurs).  Des raffinements de ce crit\`ere d\^us \`a E.~M. Jabbouri \cite{Jab} et \`a C.~Jadot \cite{Jad} permettent d'en d\'eduire des mesures d'ind\'ependance alg\'ebrique mais, dans le travail pr\'esent, nous n'aurons besoin que des propri\'et\'es les plus simples du r\'esultant de polyn\^omes homog\`enes en plusieurs variables.

Pour rappeler ces propri\'et\'es, on travaille avec l'anneau de polyn\^omes $\bC[\usX]$ o\`u $\usX=(X_0,\dots,X_t)$ est une famille de $t+1$ ind\'etermin\'ees et, pour chaque entier $D\ge 0$, on d\'esigne par $\bC[\usX]_D$ sa partie homog\`ene de degr\'e $D$, c'est-\`a-dire le sous-espace vectoriel de $\bC[\usX]$ engendr\'e sur $\bC$ par les mon\^omes de degr\'e total $D$.

\begin{proposition}
\label{propRes}
Pour chaque suite de $t+1$ entiers positifs $\uD=(D_0,\dots,D_t)$, il existe une application polynomiale
\[
 \Res_\uD\colon \bCC \fleche \bC\,,
\]
appel\'ee \emph{r\'esultant en degr\'e $\uD$}, qui poss\`ede les propri\'et\'es suivantes\ :
\begin{itemize}
\item[(1)] ses z\'eros sont les suites de polyn\^omes $(Q_0,\dots,Q_t)\in\bCC$ qui admettent au moins un z\'ero commun dans $\bP^t(\bC)$\ ;
\item[(2)] pour chaque $i=0,\dots,t$, elle est homog\`ene de degr\'e $D_0\cdots D_t/D_i$ en son argument d'indice $i$, c'est-\`a-dire que
    \[
    \Res_\uD(Q_0,\dots,\lambda Q_i,\dots,Q_t)
    = \lambda^{D_0\cdots D_t/D_i} \Res_\uD(Q_0,\dots,Q_t)
    \]
    pour tout $\lambda\in\bC$ et tout $(Q_0,\dots,Q_t)\in\bCC$\ ;
\item[(3)] le polyn\^ome qui la repr\'esente dans la base des $(t+1)$-uplets de mon\^omes du $\bC$-espace vectoriel $\bCC$ est \`a coefficients dans $\bZ$ et irr\'eductible sur $\bZ$\ ;
\item[(4)] le polyn\^ome en question est de longueur au plus $(t+1)^{3(t+1)D_0\cdots D_t}$.
\end{itemize}
\end{proposition}

En vertu du th\'eor\`eme des z\'eros d'Hilbert, les propri\'et\'es (1) et (3) caract\'erisent $\Res_\uD$ \`a multiplication pr\`es par $\pm 1$.  Les propri\'et\'es (2) et (4) viennent donc simplement pr\'eciser sa nature en majorant son degr\'e et la taille de ses coefficients dans la base des $(t+1)$-uplets de mon\^omes.  Ces estimations sont cruciales pour les applications en ind\'ependance alg\'ebrique comme l'exemple que nous allons tra\^{\i}ter ici l'illustrera.

En degr\'e $\uun=(1,\dots,1)$, le r\'esultant est un objet familier.  C'est simplement le d\'eterminant de $t+1$ formes lin\'eaires, $\det\colon(\bC[\usX]_1)^{t+1}\to\bC$.  C'est une application homog\`ene de degr\'e $1$ en chacun de ses $t+1$ arguments et son polyn\^ome sous-jacent est de longueur $(t+1)! \le (t+1)^{t+1}$.

L'existence d'une application polynomiale satisfaisant les conditions (1) \`a (3) est un r\'esultat classique \cite[Chap.~I, \S\S 7, 9, 10]{Mac}.  Quant \`a la propri\'et\'e (4), de nature arithm\'etique, elle semble \^etre plus r\'ecente et son extension aux vari\'et\'es projectives fait l'objet de recherches actives initi\'ees dans \cite{Ph4}.  L'estimation que nous donnons ici est relativement grossi\`ere mais suffisante pour notre objet.

\begin{proof}[D\'emonstration de la propri\'et\'e (4)]
La proposition 2.8 de \cite{Ph2} montre que la mesure de Mahler $M(\Res_\uD)$ du r\'esultant en degr\'e $\uD$ (d\'efinie dans \cite[\S I.3]{Ph2}) satisfait
\begin{align*}
 \log M(\Res_\uD)
 &\le D_0\cdots D_t \big((t+1)\log(t+1)+\log M(\Res_\uun)\big)\\
 &\le 2(t+1)\log(t+1)D_0\cdots D_t,
\end{align*}
la seconde estimation utilisant la majoration $M(\Res_\uun)\le L(\Res_\uun) \le (t+1)^{t+1}$ mentionn\'ee pr\'ec\'edemment.  Gr\^ace au lemme 1.13 de \cite{Ph2}, on en d\'eduit, comme dans la preuve du lemme 3.5 de \cite{LR}, que
\begin{align*}
 L(\Res_\uD)
 \le M(\Res_\uD) \prod_{i=0}^t \binom{D_i+t}{t}^{D_0\cdots D_t/D_i}
 \le (t+1)^{3(t+1)D_0\cdots D_t},
\end{align*}
en utilisant l'estimation grossi\`ere $\binom{D+t}{t}\le (t+1)^D$ valable pour tout entier $D\ge 0$.
\end{proof}

Notons que si on applique plut\^ot la proposition 5.3 de \cite{LR}, on obtient pour $L(\Res_\uD)$ la majoration $(t+1)^{4(t+1)D_0\cdots D_t}$ qui est moins pr\'ecise.

\smallskip
Soit $K\subset \bC$ un corps de nombres de degr\'e $d$, comme au paragraphe pr\'ec\'edent. De la proposition \ref{propRes}, on ne retiendra que la cons\'equence suivante.

\begin{corollaire}
\label{corRes1}
Soient $Q_0,\dots,Q_t$ des polyn\^omes homog\`enes de $K[\usX]$ de degr\'es respectifs $D_0,\dots,D_t\ge 1$, sans z\'ero commun dans $\bP^t(\bC)$, et soit $(\xi_1,\dots,\xi_t)\in\bC^t$. Pour $i=0,\dots,t$, on pose  $\delta_i := Q_i(1,\xi_1,\dots,\xi_t)$.  Alors, on a
\[
 1 \le (t+1)^{4d(t+1)D_0\cdots D_t}
        \left(\prod_{i=0}^t H(Q_i)^{dD_0\cdots D_t/D_i}\right)
        \max\left( \frac{|\delta_0|}{\|Q_0\|},\dots,\frac{|\delta_t|}{\|Q_t\|} \right).
\]
\end{corollaire}

\begin{proof}[D\'emonstration]
On a
\[
 \Res_\uD(Q_0,\dots,Q_t)\neq 0 \et \Res_\uD(Q_0-\delta_0X_0^{D_0},\dots,Q_t-\delta_tX_0^{D_t})=0
\]
puisque $Q_0,\dots,Q_t$ n'ont pas de z\'ero commun dans $\bP^t(\bC)$ tandis que, pour $i=0,\dots,t$, les polyn\^omes $Q_i-\delta_iX_0^{D_i}\in \bC[\usX]_{D_i}$ s'annulent tous au point $(1:\xi_1:\cdots:\xi_t)\in \bP^t(\bC)$.  La conclusion suit en appliquant la proposition \ref{propU} avec pour $R$ le polyn\^ome sous-jacent \`a $\Res_\uD$ et $N_i=D_0\cdots D_t/D_i$ pour $i=0,\dots,t$.
\end{proof}

En sp\'ecialisant encore davantage, on en d\'eduit le r\'esultat suivant adapt\'e sp\'ecifiquement \`a notre objet o\`u, rappelons-le, $\cO_K$ d\'esigne l'anneau des entiers de $K$.

\begin{corollaire}
\label{corRes2}
Soient $1\le D\le S$ des entiers, soit $\cF$ un sous-ensemble fini de polyn\^omes homog\`enes de $\cO_K[\usX]$ de degr\'e $S$ sans z\'ero commun dans $\bP^t(\bC)$, soit $P$ un polyn\^ome homog\`ene non nul de $K[\usX]$ de degr\'e $D$, et soit $(\xi_1,\dots,\xi_t)\in\bC^t$.  On se donne des nombres r\'eels positifs $B$ et $\delta$ tels que $\|Q^\sigma\|\le B$ et $|Q(1,\xi_1,\dots,\xi_t)|\le \delta$ pour tout $Q\in\cF$ et tout plongement $\sigma$ de $K$ dans $\bC$.  Alors on a
\[
 1 \le H(P)^{dS^t} \big( (t+1)^{8S}S^{2t}B\big)^{dtDS^{t-1}}
        \max\left\{\frac{\delta}{B},\,\frac{|P(1,\xi_1,\dots,\xi_t)|}{\|P\|} \right\}.
\]
\end{corollaire}

\begin{proof}[D\'emonstration]
Une construction g\'eom\'etrique d\^ue \`a W.~D.~Brownawell et D.~W.~Masser \cite{BM}, reprise par le lemme 1.9 de \cite{Ph2}, montre que, pour $i=1,\dots,t$, il existe une combinaison lin\'eaire $Q_i$ d'au plus $DS^{i-1}$ \'el\'ements de $\cF$, \`a coefficients entiers de valeur absolue au plus $DS^{i-1}$, telle que, pour $i=0,\dots,t$, la vari\'et\'e $V_i$ des z\'eros communs de $P,Q_1,\dots,Q_i$ dans $\bP^t(\bC)$ soit \'equidimensionnelle de dimension $t-i-1$ et, si $i\neq t$, de degr\'e au plus $DS^i$.  Pour l'\'etablir, on proc\`ede par r\'ecurrence sur $i$ en notant que, pour $i=0$, cette condition est remplie puisque $P$ est non nul de degr\'e $D$. Par ailleurs, si cette condition est satisfaite pour un entier $i$ avec $0\le i<t$ et un choix appropri\'e de $Q_1,\dots,Q_i$, alors $V_i$ poss\`ede au plus $DS^i$ composantes irr\'eductibles $W_1,\dots,W_k$ toutes de m\^eme dimension $t-i-1$.  Pour chaque $j=1,\dots,k$, on choisit un point $\uw_j$ de $W_j$, puis un \'el\'ement $F_j$ de $\cF$ tel que $F_j(\uw_j)\neq 0$.  Alors il existe une combinaison lin\'eaire $Q_{i+1}=m_1F_1+\cdots+m_kF_k$ avec $m_1,\dots,m_k$ entiers, en valeur absolue au plus $DS^i$, telle que $Q_{i+1}(\uw_j)\neq 0$ pour $j=1,\dots,k$.  Pour ce choix de $Q_{i+1}$, la vari\'et\'e $V_{i+1}$ intersection de $V_i$ et de l'hypersurface $Q_{i+1}=0$ est \'equidimensionnelle de dimension $t-i-2$ et, si $i+1\neq t$, de degr\'e au plus $\deg(V_i)S \le DS^{i+1}$ comme requis.  En particulier, pour $i=t$, cela signifie que $P,Q_1,\dots,Q_t$ n'ont aucun z\'ero commun dans $\bP^t(\bC)$.

Pour tout $i=1,\dots,t$ et tout plongement $\sigma$ de $K$ dans $\bC$, on a
\[
 \|Q_i^\sigma\|\le S^{2t}B \et |Q_i(1,\xi_1,\dots,\xi_t)|\le S^{2t}\delta
\]
car $D\le S$.  Comme chacun des $Q_i$ est \`a coefficients dans $\cO_K$, on en d\'eduit
\[
 H(Q_i)^d \le H_\infty(Q_i)^d = \prod_\sigma \|Q_i^\sigma\| \le (S^{2t}B)^{d-1}\|Q_i\| \le (S^{2t}B)^d \quad (1\le i\le t).
\]
Alors le corollaire \ref{corRes1} livre
\[
 1 \le (t+1)^{4d(t+1)DS^t}
        H(P)^{dS^t} (S^{2t}B)^{dtDS^{t-1}}
        \max\left\{ \frac{S^{2t}\delta}{S^{2t}B},\,\frac{|P(1,\xi_1,\dots,\xi_t)|}{\|P\|} \right\}
\]
et la conclusion suit en majorant le premier facteur par $(t+1)^{8dtDS^t}$.
\end{proof}

%
%

\section{Lemme de z\'eros}
\label{sec:LZ}

On reprend les notations du paragraphe \ref{sec:Fa} avec $S,T\ge 1$ fix\'es.  On pose aussi
\[
 \uX = (X_1,\dots,X_t)
 \et
 \usX = (X_0,\dots,X_t)
\]
o\`u $X_0,\dots,X_t$ sont des ind\'etermin\'ees sur $\bC$.  Enfin, pour tout $\un=(n_1,\dots,n_t)\in\bN^t$ et tout $\uz=(z_1,\dots,z_t)\in \bC^t$, on d\'efinit
\[
 \uX^\un = X_1^{n_1}\cdots X_t^{n_t}
 \et
 \uz^\un = z_1^{n_1}\cdots z_t^{n_t}.
\]
Alors, pour $\ell\in\{0,1,\dots,T-1\}$ et $\un\in\Sigma(S+1)$, la premi\`ere formule du corollaire \ref{Fa:cor:g} montre que
\begin{equation}
\label{LZ:eq:gQ}
 g^{(\ell)}(\nalpha)=Q_{\un,\ell}(1,e^{\alpha_1},\dots,e^{\alpha_t})
\end{equation}
o\`u
\begin{equation}
\label{LZ:eq:Q}
 Q_{\un,\ell}(\usX)
 = X_0^{S-|\un|}\uX^{\un}
  - \sum_\mj A_{\um,j}^{(\ell)}(\nalpha) X_0^{S-|\um|}\uX^\um
\end{equation}
est un polyn\^ome homog\`ene de $\bC[\usX]$ de degr\'e $S$.  Le but de ce paragraphe est de d\'emontrer le r\'esultat suivant qui nous permettra par la suite d'appliquer le corollaire \ref{corRes2}.

\begin{proposition}
\label{propLZ}
Supposons $T\ge 2$.  Alors les polyn\^omes $Q_{\un,\ell}$ avec $\un\in \Sigma(S+1)$ et $0\le \ell<T$ n'ont pas de z\'ero commun dans $\bP^t(\bC)$.
\end{proposition}

On pourrait d\'emontrer un r\'esultat de ce type en employant, pour le groupe $\Ga\times\Gm$, le formidable lemme de z\'eros de P.~Philippon \cite{Ph3} comme le fait M.~Ably dans \cite{Ably} mais il faudrait encore tra\^{\i}ter s\'epar\'ement les points ``\`a l'infini''.  La d\'emonstration ci-dessous est purement \'el\'ementaire et \'evite cette difficult\'e.

\begin{proof}[D\'emonstration]
Supposons au contraire qu'ils admettent un z\'ero commun $(z_0:\cdots:z_t)\in\bP^t(\bC)$ et posons $\uz=(z_1,\dots,z_t)\in \bC^t$.  Si $z_0=0$, on obtient aussit\^ot $\uz^\un=0$ pour tout $\un\in \Sigma(S+1) \setminus \Sigma(S)$, mais alors $z_1=\cdots=z_t=0$, ce qui est impossible.  Donc on doit avoir $z_0\neq 0$, ce qui permet de supposer que $z_0=1$.  Alors, en termes du polyn\^ome
\[
 p(x) := \sum_\mj A_{\um,j}(x)\uz^\um \in \bC[x],
\]
l'hypoth\`ese devient
\begin{equation}
 \label{propLZ:eq1}
 \uz^\un = p^{(\ell)}(\nalpha) \quad (\un\in\Sigma(S+1),\, 0\le \ell<T).
\end{equation}
Pour $i=1,\dots,t$, on en d\'eduit que
\[
 p^{(\ell)}(\malpha+\alpha_i)
 = z_i\uz^\um
 = z_i p^{(\ell)}(\malpha) \quad ( \um\in\Sigma(S),\ 0\le \ell <T),
\]
et par suite
\begin{equation}
 \label{propLZ:eq2}
 p(x+\alpha_i) = z_ip(x)
\end{equation}
puisque $p(x)$ est de degr\'e inf\'erieur \`a $T\,|\Sigma(S)|$. Or, pour $\un=0$, les relations \eqref{propLZ:eq1} signifient que $p^{(\ell)}(0)=1$ $(0\le \ell<T)$.  Donc, $p(x)$ est un polyn\^ome non nul de degr\'e au moins $T-1>0$ et par suite la relation \eqref{propLZ:eq2} n'est possible que si $z_i=1$ et $\alpha_i=0$ pour $i=1,\dots,t$.  Comme aucun des $\alpha_i$ n'est nul, c'est la contradiction cherch\'ee.
\end{proof}

%
%

\section{Estimations et conclusion}
\label{sec:Est}

On poursuit avec les notations des paragraphes \ref{sec:Fa} et \ref{sec:LZ}. Pour des entiers $S,T\ge 1$ fix\'es, on pose
\[
 F:=\{\ur\in\bZ^t\,;\, \ur\neq 0 \text{ et } |\ur|< S\}
 \et
 \Delta := (T-1)!\  q^{2T|F|+T} \prod_{\ur\in F} (\ralpha)^{2T}.
\]
Le but de ce paragraphe est, dans un premier temps, de montrer que les polyn\^omes $\Delta Q_{\un,\ell}$ avec $\un\in\Sigma(S+1)$ et $0\le \ell<T$ sont \`a coefficients dans $\cO_K$ et de majorer les normes de leurs  conjugu\'es.  Ensuite, on majorera leurs valeurs absolues au point $(1,e^{\alpha_1},\dots,e^{\alpha_t})$.  En appliquant le corollaire \ref{corRes2}, on sera alors en mesure de d\'emontrer le th\'eor\`eme \ref{thm:eff} sous une forme l\'eg\`erement plus g\'en\'erale.

Auparavant, on note que tout point de $F$ s'\'ecrit sous la forme $(\pm m_1,\dots,\pm m_t)$ pour un choix de $(m_1,\dots,m_t)\in \Sigma(S)\setminus\{0\}$.  Par cons\'equent, la cardinalit\'e $|F|$ de $F$ est au plus $2^t(N-1)$ o\`u $N=|\Sigma(S)|$.  Dans les calculs qui suivent, on utilisera aussi le fait que $cq\ge 1$.  Cela d\'ecoule du fait que tout \'el\'ement non nul de $\cO_K$ poss\`ede au moins un conjugu\'e de valeur absolue $\ge 1$ et les conjugu\'es de $q\alpha_1,\dots,q\alpha_t\in\cO_K$ sont tous de valeur absolue $\le cq$.

\begin{lemme}
\label{lemA}
Soient $\um\in\Sigma(S)$, $\un\in \Sigma(S+1)$ et $j,\ell\in\{0,1,\dots,T-1\}$.  Alors, les nombres $\Delta$ et $\Delta A_{\um,j}^{(\ell)}(\nalpha)$ sont des \'el\'ements de $\bZ[q\ualpha]\subseteq \cO_K$ dont les conjugu\'es ont leur valeur absolue major\'ee par
\[
 B_0 := (NT)^{2T-2}(cqS)^{2^{t+1}NT}.
\]
\end{lemme}

\begin{proof}[D\'emonstration]
En posant $E'=\Sigma(S)\setminus\{\um\}$, le polyn\^ome $A_{\um,j}(x)$ s'\'ecrit
\[
 A_{\um,j}(x) = \left(\prod_{\um'\in E'}(x-\um'\cdot\ualpha)^T\right) P_{\um,j}(x-\malpha)
\]
o\`u $P_{\um,j}$ d\'esigne le polyn\^ome de degr\'e inf\'erieur \`a $T$ pour lequel
\[
 A_{\um,j}(x)\equiv \frac{1}{j!}(x-\malpha)^j \mod (x-\malpha)^T.
\]
En rempla\c{c}ant $x$ par $x+\malpha$, cette derni\`ere condition devient
\[
 \left(\prod_{\um'\in E'}(x-(\um'-\um)\cdot\ualpha)^T\right) P_{\um,j}(x)
 \equiv \frac{x^j}{j!} \mod x^T.
\]
Ainsi, dans l'anneau des s\'eries formelles $\bC[[x]]$, on a
\begin{align*}
 P_{\um,j}(x)
  &\equiv \frac{x^j}{j!} \prod_{\um'\in E'} \left(
    \big((\um-\um')\cdot\ualpha\big)^{-T}
    \left(1-\frac{x}{(\um'-\um)\cdot\ualpha} \right)^{-T} \right) \mod x^T \\
  &\equiv \frac{x^j}{c_{\um,j}} \prod_{\um'\in E'} \left(
    \sum_{i=0}^\infty \frac{x^i}{((\um'-\um)\cdot\ualpha)^i} \right)^T  \mod x^T,
\end{align*}
o\`u $c_{\um,j}=j!\prod_{\um'\in E'} \big((\um-\um')\cdot\ualpha\big)^T$.  En d\'eveloppant cette expression, on en d\'eduit que $P_{\um,j}(x)$ est une somme de $\binom{T|E'|+T-j-1}{T-j-1}$ produits de la forme
\[
 \frac{x^{i+j}}{c_{\um,j}((\um'_1-\um)\cdot\ualpha)\cdots((\um'_i-\um)\cdot\ualpha)}
\]
o\`u $i$ est un entier avec $0\le i\le T-j-1$ et o\`u $\um'_1,\dots,\um'_i$ sont des \'el\'ements de $E'$ non n\'ecessairement distincts.   Par suite, $A_{\um,j}(x)$ est la somme des produits
\[
  \frac{1}{j!}\left( \prod_{\um'\in E'}\frac{x-\um'\cdot\ualpha}{(\um-\um')\cdot\ualpha} \right)^T \frac{(x-\malpha)^{i+j}}{((\um'_1-\um)\cdot\ualpha)\cdots((\um'_i-\um)\cdot\ualpha)}
\]
pour les m\^emes choix de $i$ et de $\um'_1,\dots,\um'_i$.  Puisque $\um-E'$ et $E'-\um$ sont des sous-ensembles de $F$, ces produits s'\'ecrivent encore sous la forme
\begin{equation}
 \label{propLZ:eq3}
 \frac{(x-\um_1\cdot\ualpha) \cdots (x-\um_k\cdot\ualpha)}
      {j!(\ur_1\cdot\ualpha) \cdots (\ur_{k-j}\cdot\ualpha)}
\end{equation}
o\`u $k$ est un entier avec $j+T|E'| \le k < T(|E'|+1)=NT$, o\`u $\um_1,\dots,\um_k$ sont des \'el\'ements de $E=\Sigma(S)$, et o\`u $\ur_1,\dots,\ur_{k-j}$ sont des \'el\'ements de $F$, chacun de ces derniers \'etant r\'ep\'et\'e au plus $2T$ fois.

En vertu de ce qui pr\'ec\`ede, la d\'eriv\'ee $A_{\um,j}^{(\ell)}(x)$ est une somme de produits semblables aux produits \eqref{propLZ:eq3} mais avec $\ell$ facteurs en moins au num\'erateur.  De plus, le nombre de ces produits ne d\'epasse pas
\[
 (NT)^\ell \binom{T|E'|+T-j-1}{T-j-1}
 \le \frac{(NT)^{2(T-1)}}{(T-1)!}\,.
\]
Alors $\Delta A_{\um,j}^{(\ell)}(\nalpha)$ est une somme d'autant de produits de la forme
\[
 \frac{(T-1)!}{j!} q^{2T|F|+T} (\ur_1\cdot\ualpha) \cdots (\ur_k\cdot\ualpha)
\]
o\`u cette fois $k$ d\'esigne un entier au plus \'egal \`a $2T|F|+T$ et o\`u $\ur_1,\dots,\ur_k$ sont des \'el\'ements de $\bZ^t$ de longueur $|\ur_i| \le S$ $(1\le i\le k)$.  On conclut que $\Delta A_{\um,j}^{(\ell)}(\nalpha)$ est un \'el\'ement de $\bZ[q\ualpha]$ et que, pour tout plongement $\sigma$ de $K$ dans $\bC$, on a
\[
 \left|\left(\Delta A_{\um,j}^{(\ell)}(\nalpha)\right)^\sigma\right|
 \le (NT)^{2(T-1)} (cqS)^{2T|F|+T}.
\]
Comme cette majoration s'applique aussi \`a $|\Delta^\sigma|$, la conclusion suit en notant que
$|F|\le 2^t(N-1)$.
\end{proof}

\begin{proposition}
\label{propQ}
Supposons $N\ge 6$ et $NT\ge 2ecqS$.  Soient $\un\in\Sigma(S+1)$ et $\ell\in\{0,1,\dots,T-1\}$.  Alors, $\Delta Q_{\un,\ell}$ est un polyn\^ome homog\`ene de $\cO_K[\usX]$ de degr\'e $S$ qui satisfait, pour tout plongement $\sigma$ de $K$ dans $\bC$,
\[
 \|\Delta^\sigma Q_{\un,\ell}^\sigma\| \le TB_0
 \et
 |\Delta Q_{\un,\ell}(1,e^{\alpha_1},\dots,e^{\alpha_t})| \le \left(\frac{cqS}{T}\right)^{NT} T^T B_0.
\]
\end{proposition}

\begin{proof}[D\'emonstration]
La majoration de norme ainsi que l'assertion qui la pr\'ec\`ede d\'ecoulent directement du lemme pr\'ec\'edent joint \`a la d\'efinition \eqref{LZ:eq:Q} du polyn\^ome $Q_{\un,\ell}$.  Pour \'etablir l'autre majoration, on utilise la seconde formule du Corollaire \ref{Fa:cor:g}
\begin{align*}
 \Delta Q_{\un,\ell}(1,e^{\alpha_1},\dots,e^{\alpha_t})
  &= \Delta g^{(\ell)}(\nalpha) \\
  &= \sum_{k=NT}^\infty \frac{1}{k!}
    \bigg( \Delta k^{(\ell)}(\un\cdot\ualpha)^{k-\ell} - \sum_\mj k^{(j)} (\um\cdot\ualpha)^{k-j}
           \Delta A^{(\ell)}_{\um,j}(\un\cdot\ualpha) \bigg).
\end{align*}
Gr\^ace au lemme pr\'ec\'edent, la valeur absolue de cette quantit\'e est major\'ee par
\[
B_0 \sum_{k=NT}^\infty \frac{1}{k!}
    \bigg( k^{T-1}(cS)^{k-\ell} + N\sum_{j=0}^{T-1} k^j (cS)^{k-j} \bigg)
    \le B_0 \sum_{k=NT}^\infty \frac{(cqS)^k}{k!} Nk^T,
\]
puisque la somme porte sur $k\ge NT$, que $NT\ge 3$, et que $\max\{1,cS\} \le cqS$.  Comme $k!\ge (k/e)^k$ pour tout entier $k\ge 0$, cette derni\`ere quantit\'e est \`a son tour major\'ee par
\begin{align*}
N B_0 \sum_{k=NT}^\infty \left(\frac{ecqS}{k}\right)^k k^T
 &\le N B_0 \sum_{k=NT}^\infty \left(\frac{ecqS}{NT}\right)^k (NT)^T \\
 &\le 2N B_0 \left(\frac{ecqS}{NT}\right)^{NT} (NT)^T &&\text{car $NT\ge 2ecqS$,}\\
 &\le \left(\frac{cqS}{T}\right)^{NT} T^T B_0 &&\text{car $N\ge 6$.}
\end{align*}
\end{proof}

Si $D$, $H$ et $P$ sont comme au th\'eor\`eme \ref{thm:eff}, alors le polyn\^ome homog\`ene $\tP$ de degr\'e $D$ associ\'e \`a $P$ est de hauteur $H(\tP)\le H$ et satisfait par construction $\tP(1, e^{\alpha_1}, \dots, e^{\alpha_t}) = P(e^{\alpha_1}, \dots, e^{\alpha_t})$.  Donc le th\'eor\`eme \ref{thm:eff} d\'ecoule du r\'esultat suivant sur lequel nous concluons ce travail.

\begin{theoreme}
\label{thm:K}
Soit $D\ge 1$ un entier positif et soit $P$ un polyn\^ome homog\`ene non nul de $K[\usX]$ de degr\'e $D$.  Alors, on a
\begin{equation}
\label{thm:K:eq0}
 \rho:=\frac{|P(1,e^{\alpha_1},\dots,e^{\alpha_t})|}{\|P\|}
 \ge H(P)^{-3dS^t}\exp\left(-(cqS)^{18S^t}\right),
\end{equation}
o\`u $S=6dt(t!)D$.
\end{theoreme}

\begin{proof}[D\'emonstration]
Soit $T$ le plus petit entier positif qui satisfasse \`a la fois
\[
 T \ge (cqS)^{2^{t+2}N} \et T^{NT} \ge H(P)^{3dS^t}, \quad\text{o\`u}\quad N=|\Sigma(S)|=\binom{S+t-1}{t}.
\]
En vertu des propositions \ref{propLZ} et \ref{propQ}, les produits $\Delta Q_{\un,\ell}$ avec $\un\in\Sigma(S+1)$ et $0\le \ell<T$ forment une famille $\cF$ de polyn\^omes homog\`enes de $\cO_K[\usX]$ de degr\'e $S$ sans z\'ero commun dans $\bP^t(\bC)$ qui satisfont les conditions du corollaire \ref{corRes2} avec $B=TB_0$ et $\delta=(cqS/T)^{NT} T^T B_0$ o\`u $B_0$ est donn\'e par le lemme \ref{lemA}.  Ce corollaire livre donc
\begin{equation}
\label{thm:K:eq1}
 1 \le H(P)^{dS^t} \Big((t+1)^{8S}S^{2t}TB_0\Big)^{dtDS^{t-1}}
      \max\left\{ \left(\frac{cqS}{T}\right)^{NT} T^T, \ \rho\right\}\,.
\end{equation}
Pour les estimations subs\'equentes, on utilisera
\[
 \frac{S^t}{t!}
  \le N=\frac{S(S+1)\cdots(S+t-1)}{t!}
  \le \frac{(S+t)^t}{t!}
  \le \frac{S^t}{t!}\left(1+\frac{1}{6t}\right)^t
  \le \frac{2S^t}{t!},
\]
qui impliquent plus grossi\`erement $6t\le S \le N\le S^t$ en tenant compte du fait que, pour $t=1$, on a $N=S$. Ainsi, en vertu du choix de $S$, on trouve
\[
 dtDS^{t-1} \le \frac{S^t}{6t!}\le \frac{N}{6}.
\]
On va montrer que
\begin{equation}
\label{thm:K:eq2}
 H(P)^{dS^t} \Big((t+1)^{8S}S^{2t}TB_0\Big)^{N/6}
      \left(\frac{cqS}{T}\right)^{NT} T^T \ < \ 1.
\end{equation}
Si on accepte ce r\'esultat, l'in\'egalit\'e \eqref{thm:K:eq1} entra\^{\i}ne aussit\^ot
\begin{equation}
\label{thm:K:eq3}
 \rho > \left(\frac{cqS}{T}\right)^{NT} T^T \ge \left(\frac{2}{T}\right)^{NT-T} \ge (T-1)^{-N(T-1)},
\end{equation}
car $cqS\ge S\ge 2$ et $T\ge N\ge 2$.  En vertu du choix de $T$, on a par ailleurs
\[
 T-1 \le (cqS)^{2^{t+2}N}\le \exp\big(2^{t+2}N(cqS)\big) \ou (T-1)^{N(T-1)} \le H(P)^{3dS^t}.
\]
Comme $H(P)\ge 1$, on en d\'eduit
\begin{equation}
\label{thm:K:eq4}
 (T-1)^{N(T-1)} \le H(P)^{3dS^t} \exp\Big(2^{t+2}N(cqS)N(cqS)^{2^{t+2}N}\Big).
\end{equation}
Comme $2\le S$, $N\le S^t$ et $S\le cqS$, on a aussi
\begin{align*}
 2^{t+2}N(cqS)N(cqS)^{2^{t+2}N}
 &\le (cqS)^{2^{t+2}N+3t+3}\\
 &\le (cqS)^{(2^{t+2}+t)N} &&\text{car $N\ge 6$,}\\
 &\le (cqS)^{2(2^{t+2}+t)S^t/t!} &&\text{car $N\le 2S^t/t!$,}\\
 &\le (cqS)^{18S^t}.
\end{align*}
En substituant cette estimation dans \eqref{thm:K:eq4} puis le r\'esultat dans \eqref{thm:K:eq3}, on obtient l'in\'egalit\'e annonc\'ee \eqref{thm:K:eq0}.  Il reste donc \`a v\'erifier \eqref{thm:K:eq2}.  Pour ce faire, on utilise d'abord $t+1\le S$, $N\le S^t$ et $S\le cqS$, donc
\begin{align*}
 (t+1)^{8S}S^{2t}TB_0
   &= (t+1)^{8S}S^{2t}T (NT)^{2T-2}(cqS)^{2^{t+1}NT}\\
   &\le T^{2T} (cqS)^{2^{t+1}NT + 8S + 2tT} \\
   &\le T^{2T} (cqS)^{(2^{t+1}+1)NT},
\end{align*}
o\`u la derni\`ere majoration d\'ecoule du fait que $8S+2tT\le 6tT \le NT$.  Comme $N\ge 6$, on en d\'eduit que le membre de gauche de \eqref{thm:K:eq2} est strictement major\'e par
\[
 H(P)^{dS^t} \Big(T^{3T} (cqS)^{(2^{t+1}+2)NT}\Big)^{N/6} T^{-NT}.
\]
En vertu du choix de $T$, ce nombre est \`a son tour major\'e par $T^{NT/3}(T^{4T})^{N/6}T^{-NT}=1$, comme annonc\'e.
\end{proof}

\subsection*{Remerciements}
 Il me fait plaisir de remercier Daniel Bertrand, Michel Laurent et le referee pour leurs suggestions des r\'ef\'erences bibliographiques \cite{An,Be,BR,Ma2,Wu}.  Je tiens aussi \`a t\'emoigner de ma reconnaissance envers Michel Waldschmidt pour son int\'er\^et dans mes travaux, ses constants encouragements, et de si nombreuses discussions depuis plus de 20 ans d\'ej\`a.

%
%

\renewcommand*{\refname}{Bibliographie}

\end{document}